\documentclass[12pt]{amsart}  

\usepackage{amssymb}
\usepackage{hyperref}
\newtheorem{theorem}{Theorem}
\newtheorem{lemma}[theorem]{Lemma}
\newtheorem{proposition}[theorem]{Proposition}
\newtheorem{example}[theorem]{Example}
\newtheorem{corollary}[theorem]{Corollary}
\newtheorem{remar}[theorem]{Remark}
\newenvironment{remark}{\begin{remar}\rm}{\end{remar}}
\newcommand{\bfind}[1]{\index{#1}{\bf #1}}
\newcommand{\n}{\par\noindent}

\newcommand{\sn}{\par\smallskip\noindent}

\newcommand{\bn}{\par\bigskip\noindent}
\newcommand{\pars}{\par\smallskip}
\newcommand{\parm}{\par\medskip}

\newcommand{\ac}{^{\rm ac}}

\newcommand{\ovl}[1]{\overline{#1}}

\newcommand{\sep}{^{\rm sep}}
\newcommand{\Aut}{\mbox{\rm Aut}\,}
\newcommand{\res}{\mbox{\rm res}\,}
\newcommand{\chara}{\mbox{\rm char}\,}

\newcommand{\ic}{\mbox{\rm IC}\,}

\newcommand{\diams}{\unskip\nobreak\hfil\penalty50%
\hskip1em\hbox{}\nobreak\hfil%
$\diamondsuit$\parfillskip=0pt\finalhyphendemerits=0}

\newcommand{\Q}{\mathbb Q}

\newcommand{\Z}{\mathbb Z}
\newcommand{\F}{\mathbb F}
\newcommand{\cal}{\mathcal}

\newcommand{\cO}{{\mathcal O}}

\begin{document}
\title[Eliminating Tame Ramification]{Eliminating Tame Ramification: generalizations of Abhyankar's Lemma
}

\author{Arpan Dutta and Franz-Viktor Kuhlmann}

\address{Department of Mathematics, IISER Mohali,
	Knowledge City, Sector 81, Manauli PO,
SAS Nagar, Punjab, India, 140306.}
\email{arpan.cmi@gmail.com}

\address{Institute of Mathematics, ul.~Wielkopolska 15, 70-451 Szczecin, Po\-land}
\email{fvk@usz.edu.pl}

\date{6.\ 2.\ 2020}

\thanks{The first author would like to thank Steven Dale Cutkosky and Sudesh K. Khanduja for their support, suggestions and helpful discussions.\\
A preliminary part of this paper was written during Research in Teams at the Banff
International Research Station in 2003. The second author is very grateful for having
been able to use this great facility. He would also like to
thank Hagen Knaf for inspiring discussions.\n
Further, the authors thank the referee for many helpful suggestions and corrections.\n
The second author was partially supported by a Canadian NSERC grant and is currently supported by Opus grant 2017/25/B/ST1/01815 from the National Science Centre of Poland.}


\begin{abstract}
A basic version of Abhyankar's Lemma states that for two finite extensions $L$ and $F$ of a local field $K$,
if $L|K$ is tamely ramified and if the ramification index of $L|K$ divides the ramification index of $F|K$, then
the compositum $L.F$ is an unramified extension of $F$. In this paper, we generalize the result to valued fields
with value groups of rational rank 1, and show that the latter condition is necessary. Replacing the condition
on the ramification indices by the condition that the value group of $L$ be contained in that of $F$, we
generalize the result further in order to give a necessary and sufficient condition for the elimination of tame
ramification of an arbitrary extension $F|K$ by a suitable algebraic extension of the base field $K$. In
addition, we derive more precise ramification theoretical statements and give several examples.
\end{abstract}

\maketitle
%
%
%
%
\section{Introduction}
In this paper we consider valued fields $(K,v)$, i.e., fields $K$ with a Krull valuation $v$. The valuation
ring of $v$ on $K$ will be denoted by ${\cal O}_K\,$. The value group of $(K,v)$ will be denoted by $vK$, and
its residue field by $Kv$. The value of an element $a$ will be denoted by $va$, and its residue by $av$.
By $(L|K,v)$ we denote a field extension $L|K$ where $v$ is a valuation on $L$ and $K$ is endowed with the
restriction of $v$. For background on valuation theory, see \cite{En,EP,Kubook,ZS2}. Basic facts that we will
need, in particular from ramification theory, will be presented in Section~\ref{sectprel}.

Throughout, we will consider the following general situation. We let $(M,v)$ be an arbitrary algebraically closed
extension of some valued field $(K,v)$. Every subfield $E$ of $M$ will be endowed with the restriction of $v$,
which we will again denote by $v$; note that $(M,v)$ contains a unique henselization of $(E,v)$, which we
denote by $(E^h,v)$. Further, we take an arbitrary subextension $F|K$ and an algebraic subextension $L|K$ of $M|K$.
The \bfind{compositum of the fields} $F$ and $L$ within $M$ is the smallest subfield of $M$ that contains both $F$
and $L$, and we denote it by $L.F\,$. The restriction of $v$ from $M$ to $L.F$ is then a simultaneous extension
of the restrictions to $L$ and $F$. Similarly, the \bfind{compositum of the value groups} $vF$ and $vL$ within
$vM$ is the smallest subgroup of $vM$ that contains both $vF$ and $vL$, and we denote it by $vL+vF$.

An algebraic extension $(L|K,v)$ of henselian fields is called \bfind{tame}
if every finite subextension $E|K$ of $L|K$ satisfies the following conditions:
\sn
(TE1) the ramification index $(vE:vK)$ is not divisible by $\chara Kv$.
\n
(TE2) the residue field extension $Ev|Kv$ is separable.
\n
(TE3) the extension $(E|K,v)$ is \bfind{defectless}, i.e.,
\[
[E:K]\>=\>(vE:vK)[Ev:Kv]\>.
\]
Note that the extension $(L|K,v)$ is called \bfind{tamely ramified} if (TE1) and (TE2) hold for all
finite subextensions $E|K$, so a finite tame extension is the same as a finite defectless tamely ramified
extension. The extension $(L|K,v)$ is called \bfind{unramified} if the canonical embedding of $vK$ in $vL$ is
onto and the residue field extension $Lv|Kv$ is separable; this does not necessarily imply that the extension is
defectless.

\pars
In the case of a henselian discretely valued field $(K,v)$, condition (TE3) is known to hold as soon as
$L|K$ is separable. Therefore, if in addition $\chara K=0$, then a finite extension of $(K,v)$ is tame once it
is tamely ramified. If in addition $(K,v)$ is complete, then condition (TE3) always holds.

\pars
For henselian discretely valued fields, Abhyankar's Lemma provides a sufficient condition to eliminate tame
ramification of a finite extension $(F|K,v)$ by lifting through a finite extension. In this case we can choose $M$
to be the algebraic closure of $K$, and the extension of $v$ from $K$ to $L$, $F$ and $L.F$ is uniquely determined.
\begin{theorem}\textbf{(Abhyankar's Lemma)}\label{Abhyankar lemma}
Let $(K,v)$ be a henselian discretely valued field, $(L|K,v)$ be a finite tame extension and $(F|K,v)$ a finite
extension. If the ramification index of $(L|K,v)$ divides the ramification index of $(F|K,v)$, then
the extension $(L.F/F,v)$ is unramified.
\end{theorem}

In \cite{CH} the following version of Abhyankar's Lemma is shown: the ramification index of the compositum of two
finite extensions of local fields is equal to the least common multiple of the ramification indices corresponding
to the finite extensions, provided at least one of the extensions is tame. This version is a special case of a
more general theorem that we will present next.

\pars
The condition on the ramification indices in Theorem~\ref{Abhyankar lemma} is also necessary. Indeed, $(L.F|F,v)$
being unramified implies that $v(L.F) = vF$. Thus,
\[
(vF:vK) \>=\> (v(L.F) : vK) \>=\> (v(L.F): vL) (vL:vK)\>,
\]
hence $(vL:vK)$ divides $(vF:vK)$.

\pars
The question naturally arises how far the above formulation of Abhyankar's Lemma can be generalized. The next
theorem, which implies Theorem~\ref{Abhyankar lemma}, shows that the result remains true whenever $vK$ has
rational rank 1; the rational rank of an abelian group is the $\Q$-dimension of the divisible hull $\Q\otimes_\Z
\Gamma$ of $\Gamma$.

\pars
From now on we will assume the general situation as introduced in the beginning, i.e., $F|K$ is an arbitrary
extension, and $L|K$ is a (not necessarily finite) algebraic extension.
\begin{theorem}                            \label{GALrr1}
Assume that the value group of $(K,v)$ is of rational rank 1, that the extension $(L.K^h|K^h,v)$ is tame and that
the ramification indices $(vL:vK)$ and $(vF:vK)$ are finite. Then $(v(L.F):vK)$ is the least common multiple of
$(vL:vK)$ and $(vF:vK)$. In particular, $(L.F|F,v)$ is unramified if and only if the
ramification index of $(L|K,v)$ divides the ramification index of $(F|K,v)$.
\end{theorem}

In contrast, in Section~\ref{sectALri} we will show that the result fails for higher rational rank (see
Lemma~\ref{2ri}). In particular, the result fails for generalized discretely valued fields, i.e., those valued
fields whose value group is a lexicographically ordered product of finitely many copies of $\Z$.

\parm
By reformulating the condition on the ramification indices in a different way, using the value groups themselves
instead, one can prove a far-reaching generalization of Abhyankar's Lemma. The {\bf absolute ramification field
$(K^r,v)$ of} $(K,v)$ is the ramification field of the normal extension $(K\sep|K,v)$, where $K\sep$ denotes the
separable-algebraic closure of $K$. Likewise, the {\bf absolute inertia field $(K^i,v)$ of} $(K,v)$ is the inertia
field of the extension $(K\sep|K,v)$. Since $M$ is assumed to be algebraically closed,  just as for
henselizations, it contains a unique ramification field and a unique inertia field for every subfield $(E,v)$.
We have that $E^h\subseteq E^i\subseteq E^r$ and hence, $(E^i,v)$ and $(E^r,v)$ are henselian.

An extension $(L|K,v)$ of valued fields is called \bfind{immediate} if the canonical embeddings of $vK$ in $vL$
and of $Kv$ in $Lv$ are onto. Recall that the henselization is an immediate extension.

\pars
In Section~\ref{sectpfGAL}, we will prove the following:
\begin{theorem}                             \label{GAL}
1) Assume that $(L,v)$ is contained in the absolute ramification field of $(K,v)$. Then $(L.F,v)$ is contained in
the absolute ramification field of $(F,v)$ and $v(L.F)=vL+vF$. Further, $(L.F,v)$ is contained in the
absolute inertia field of $(F,v)$ (which implies that the extension $(L.F|F,v)$ is unramified)
if and only if $vL$ is a subgroup of $vF$.
\sn
2) Assume that $(L,v)$ is contained in the absolute inertia field of $(K,v)$. Then $(L.F,v)$ is contained in the
absolute inertia field of $(F,v)$ and $(L.F)v=Lv.Fv$. Further, $(L.F,v)$ is contained in the
henselization of $(F,v)$ (which implies that the extension $(L.F|F,v)$ is immediate)
if and only if $Lv$ is a subfield of $Fv$.
\end{theorem}

In Section~\ref{sectALri} we will show that this theorem implies Theorem~\ref{GALrr1} and hence also
Theorem~\ref{Abhyankar lemma}.

\pars
Note that if $\chara Kv=0$, then the absolute ramification field is
algebraically closed, so $(L,v)$ is contained in it as soon as
$L|K$ is algebraic. If $\chara Kv>0$ and $L|K$ is algebraic, then for
$(L,v)$ to lie in the absolute ramification field $(K^r,v)$ of $(K,v)$,
the following three conditions are necessary and sufficient (the letters ``PT'' stand
for ``pre-tame''):
\sn
(PT1) \ $\chara Kv$ does not divide the order of any non-zero element in $vL/vK$,
\n
(PT2) \ the residue field extension $Lv|Kv$ is separable,
\n
(PT3) \ for every finite subextension $E|K$ of $L|K$, the extension\n
$(E^h|K^h,v)$ of their respective henselizations (in $(M,v)\,$) is defectless.
\sn
This means that if $(K,v)$ is henselian, then $(L,v)$ lies in its absolute ramification field if and only if
$(L|K,v)$ is a tame extension; in other words, $(K^r,v)$ is the unique maximal tame extension of $(K,v)$.

Similarly, $(L,v)$ lies in the absolute inertia field of $(K,v)$ if and
only if $L|K$ is algebraic, $vL=vK$, and conditions (PT2) and (PT3) hold.

\pars
Assume now that $\chara Kv=p>0$. Does elimination of tame ramification also hold if the extension $(L^h|K^h,v)$
is not tame? The answer is yes
if we restrict the scope to normal extensions. We denote by $(vL)_{p'}$ the maximal subgroup of $vL$ containing
$vK$ and such that $p$ does not divide the order of any of its nonzero element modulo $vK$. Further,
we denote by $(Lv)_s$ the maximal subfield of $Lv$ separable over $Kv$. A $p$-extension is a (not necessarily
finite) Galois extension with Galois group a $p$-group.
\begin{theorem}                             \label{MT2}
Assume that $L|K$ is normal, $F|K$ is an arbitrary extension, and $\chara Kv=p>0$.
Then the following assertions hold.
\sn
1) The quotient group $v(L.F)/((vL)_{p'}+vF)$ is a $p$-group.
In particular, $v(L.F)/vF$ is a $p$-group if and only if $(vL)_{p'}\subseteq vF$.
\n
2) If $(vL)_{p'}=vK$, then the maximal separable subextension of $(L.F)v\,|$ $(Lv)_s.Fv$ is a $p$-exten\-sion.

\end{theorem}

\parm
Trivial examples of ramification that can easily be eliminated appear when the base field $K$ is smaller than the
constant field of the function field $F$. More sophisticated examples will therefore present situations where the
base field $K$ is equal to the constant field, i.e., is relatively algebraically closed in $F$. But this does
not imply that $K$ is equal to the relative algebraic closure of $K$ in a fixed henselization of $(F,v)$. In
\cite{Kurf}, for valued rational function fields $(K(x)|K,v)$ the \bfind{implicit constant field} $\ic(K(x)|K,v)$
is defined to be the relative algebraic closure of $K$ in a fixed henselization of $(K(x),v)$. While it depends
on the chosen henselization, it is unique up to valuation preserving isomorphism over $K$. The following is
Theorem~1.3 of \cite{Kurf}:
\begin{theorem}                             \label{MTIC}
Let $(L|K,v)$ be a countably generated separable-algebraic extension of
non-trivially valued fields. Then there is an extension of $v$ from $L$ to the algebraic closure $L(x)\ac=
K(x)\ac$ of the rational function field $K(x)$ such that, upon taking henselizations in $(K(x)\ac,v)$,
\[
L^h\>=\>  \ic(K(x)|K,v)\>.
\]
\end{theorem}
\n
This means that $L\subset K(x)^h$, so that $L(x)=L.K(x)$ lies in the henselization of $K(x)$ and all ramification,
whether tame or wild, is eliminated. We will construct specific examples in Section~\ref{sectex}.

\parm
Finally, let us mention that there are various other versions and generalizations of Abhyankar's Lemma. Here we
list only a few. When the valued field $(K,v)$ is a formally $\wp$-adic field, then
Theorem~\ref{Abhyankar lemma} is Corollary~4 in \cite[Chapter 5]{N}. Elimination of ramification by so-called
strongly solvable extensions of the base field has been presented in \cite{P1,P2}. Generalizations are also
discussed in the Stacks Project \cite{SP}, some of which we will cite in Section~\ref{sectALri}. Finally, a
``perfectoid Abhyankar lemma'' has recently been presented in \cite{An}.

\bn
%
%
\section{Preliminaries}                                   \label{sectprel}
We recall some aspects of ramification theory and of general valuation theory [cf.\ e.g.\
\cite{Abh,En,EP,Kubook,Neu,ZS2}. We take a normal algebraic
extension $(L|K,v)$ of valued fields and set $G= \Aut L|K$. The \textbf{decomposition group} of the extension is
defined as
\[
G^d(L|K,v) \>:=\> \{ \sigma \in G \mid v \circ \sigma = v \text{ on } L \}\>,
\]
the \textbf{inertia group} as
\[
G^i(L|K,v) \>:=\> \{ \sigma \in G \mid \forall \, x \in \cO_L: v(\sigma x - x ) > 0\}\>,
\]
and the \textbf{ramification group} as
\[
G^r(L|K,v) \>:=\> \{ \sigma \in G \mid \forall \, x \in L^\times: v(\sigma x - x ) > vx \}\>.
\]
The corresponding fixed fields in $K\sep$ will be denoted as $(L|K,v)^d$, $(L|K,v)^i$ and $(L|K,v)^r$ and are
called the \bfind{decomposition field}, \bfind{inertia field} and \bfind{ramification field} of $(L|K,v)$,
respectively. We have:
\[
G^r(L|K,v) \trianglelefteq G^i(L|K,v) \trianglelefteq G^d(L|K,v) \leq G
\]
and
\[
G^r(L|K,v)\trianglelefteq G^d(L|K,v)\>,
\]
so $(L|K,v)^d\subseteq (L|K,v)^i\subseteq (L|K,v)^r$ with both extensions as well as $(L|K,v)^d\subseteq
(L|K,v)^r$ Galois.

In the above notation, the absolute decomposition field, absolute inertia field and absolute ramification field
of $(K,v)$ that we mentioned in the introduction are $K^d=(K\ac|K,v)^d=(K\sep|K,v)^d$, $K^i=(K\ac|K,v)^i=
(K\sep|K,v)^i$ and $K^r=(K\ac|K,v)^r=(K\sep|K,v)^r$, respectively.

\pars
We collect the main facts of ramification theory that we will need in this paper in the next theorem. To simplify
notation, we set $L_d=(L|K,v)^d$, $L_i=(L|K,v)^i$, $L_r=(L|K,v)^r$, and denote by $L_s$ the maximal separable
extension of $K$ inside of $L$.
\begin{theorem}                                    \label{brt}
1) The extension $(L_d|K,v)$ is immediate and $v$ has a unique extension from $L_d$ to $L$.
\sn
2) The extension $L_i v|L_d v$ is separable, and $L_r v=L_i v\,$.
\sn
3) We have that $vL_i=vL_d$, and the order of no element in $vL_r/vL_i=vL_r/vK$ is divisible by $\chara Kv$.
\sn
4) If $\chara Kv=p>0$, then $G^r(L|K,v)$ is a $p$-group, so $L_s|L_r$ is a $p$-extension. If $\chara Kv=0$, then
$G^r(L|K,v)$ is trivial and $L_r=L$. The extension $Lv|L_r v$ is purely inseparable, and $vL/vL_r$ is a $p$-group.
\sn
5) If $K\subseteq K_1\subseteq K_2\subseteq L_r$, $K_2|K_1$ is finite and $(K_1,v)$ (and thus also $(K_2,v)\,$) is
henselian, then the extension $(K_2|K_1,v)$ is defectless.
\sn
6) We have that $(L|L_d,v)^i=L_i$ and $(L|L_d,v)^r=(L|L_i,v)^r=L_r\,$.
\sn
7) If $K\subseteq L'\subseteq L$, then $(L|L',v)^d=L'.L_d\,$, $(L|L',v)^i=L'.L_i$ and $(L|L',v)^r=L'.L_r\,$.
\sn
8) Whenever $F|K$ is an arbitrary extension and the valuation $v$ is fixed on some field containing the
algebraic closure of $F$, then $K^d\subseteq F^d$, $K^i\subseteq F^i$ and $K^r\subseteq F^r$.
\sn
9) If $K\subseteq K_1\subseteq K^d$, then $K_1^d=K^d$. If $K\subseteq K_1\subseteq K^i$, then $K_1^i=K^i$.
If $K\subseteq K_1\subseteq K^r$, then $K_1^r=K^r$.
\end{theorem}

\begin{corollary}                            \label{cordl}
If $K\subseteq K_1\subseteq K'_1\subseteq L_r$, $(K'_1|K_1,v)$ is immediate and $(K_1,v)$ (and thus also
$(K'_1,v)\,$) is henselian, then $K_1=K'_1\,$.
\end{corollary}
\begin{proof}
Take $K_2|K_1$ to be any finite subextension of $K'_1|K_1\,$. Since $(K'_1|K_1,v)$ is immediate by assumption,
the same holds for $(K_2|K_1,v)$. As this extension is also defectless by part 5) of Theorem~\ref{brt}, we have
that $[K_2:K_1]=(vK_2:vK_1)[K_2v:K_1v]=1$, whence $K_1=K_2\,$. It follows that $K_1=K'_1\,$.
\end{proof}

Here is a crucial lemma for the proof of Theorems~\ref{GAL} and~\ref{MT2}:
\begin{lemma}                               \label{l}
Take any extension $(L,v)$ of $(K,v)$, elements $\beta\in vL$, $c\in K$
and a positive integer $n$ such that $n\beta=vc$. Suppose that $p$ does
not divide $n$. Then the polynomial $X^n-c$ splits in the absolute
inertia field $L^i$ of $(L,v)$ and $\beta\in vL^i$.
\end{lemma}
\begin{proof}
Take some $b\in L$ such that $vb=\beta$. Then $vcb^{-n}=0$ and
therefore, $cb^{-n}v\ne 0$. Since $p$ does not divide $n$, the
polynomial $X^n-cb^{-n}v$ has $n$ distinct roots in $(Lv)\sep=
L^iv$. By Hensel's Lemma, it follows that the polynomial $X^n-cb^{-n}$
splits completely in the henselian field $(L^i,v)$. Hence, so does $X^n-c$.
\end{proof}

\pars
Further, we will need the \bfind{fundamental inequality}, of which we state only a simple form here:
for every finite extension $(L|K,v)$,
\begin{equation}                     \label{fi}
[L:K]\>\geq\> (vL:vK)[Lv:Kv]\>.
\end{equation}

\pars
Finally, we will need:
\begin{proposition}                         \label{p-ext}
Take any prime $p$ and an arbitrary extension $F|K$ and a normal algebraic extension $L|K$. If the maximal
separable subextension of $L|K$ is a $p$-extension, then the same holds for $L.F|F$.
\end{proposition}
\begin{proof}
Let $L_s|K$ be the maximal separable subextension of $L|K$ and set $E:=L_s\cap F$. Then both $L_s|K$ and $L_s|E$
are normal and separable, and $\Aut L_s|E$ is a subgroup of $\Aut L_s|K$. Since the latter is a $p$-group by
assumption, so is the former.

Since $L_s\cap F =E$ and $L_s|E$ is normal and separable, $F$ and $L_s$ are linearly disjoint over $E$ and it
follows that $\Aut L_s.F|F=\Aut L_s|E$, which shows that $L_s.F|F$ is a $p$-extension. Since $L|L_s$ is purely
inseparable, also $L.(L_s.F)=L.F$ is a purely inseparable extension of $L_s.F$, so $L_s.F|F$ is the maximal
separable subextension of $L.F|F$.
\end{proof}

\bn
%
%
\section{Proof of Theorem~\ref{GAL}}            \label{sectpfGAL}
In this and the next two sections, we will freely use the facts collected in Theorem~\ref{brt} as well as the
fundamental inequality (\ref{fi}) without citing them.

\pars
We assume the extensions $(F|K,v)$ and $(L|K,v)$ to be as in the introduction. Since $L|K$ is algebraic, $vL/vK$
is a torsion group.

Let us first assume that $vL\subseteq vF$ and that $(L,v)$ is contained in the absolute ramification field
$K^r$ of $(K,v)$, so $vL\subseteq vK^r$. Take any set $\{\beta_j\mid j\in J\}$ of generators of
$vL$ over $vK$, and let $n_j$ be positive integers such that $n_j\beta_j\in vK$ for each $j\in J$. Since
$\chara Kv$ does not divide the order of any element in $vK^r/vK$, the same holds for
$vL/vK$. Therefore, we can assume that $\chara Kv$ does not divide any of the $n_j\,$. Applying Lemma~\ref{l},
we can find elements $b_j\in L^i$ such that $vb_j=\beta_j$ and $c_j:=b_j^{n_j}\in K$.
%
Since $K^i\subseteq L^i$, we obtain that
\[
vL\>\subseteq\> vK^i(b_j \mid j\in J)\>\subseteq\> vL^i\>=\>vL\>,
\]
showing that equality must hold everywhere. Since $Lv|Kv$ is separable
by condition (TE2), we have that $K^iv = (Kv)\sep=(Lv)\sep=L^iv$ and thus,
\[
K^iv\>\subseteq\> K^i(b_j \mid j\in J)v\>\subseteq\> L^iv\>=\>K^iv\>,
\]
showing again that equality must hold everywhere. We have proved that
\[
(L^i|K^i(b_j \mid j\in J),v)
\]
is an immediate extension.

By assumption, $(L,v)$ is an extension of $(K,v)$ within the absolute ramification field $(K^r,v)$ of $(K,v)$.
Hence also $(L^i,v)$ is contained in $(K^r,v)$. Therefore, we can apply Corollary~\ref{cordl} to find that
\[
L^i\>=\>K^i(b_j \mid j\in J)\>.
\]

Since $K\subseteq F$, it follows that $K^i\subseteq F^i$.
Since $\beta_j\in vL\subseteq vF$, we know from Lemma~\ref{l} that
the polynomials $X^{n_j}-c_j$ split completely over $F^i$. Consequently,
we also have $b_j\in F^i$ for each $j\in J$. This yields that
\[
L\>\subseteq\> L^i\>=\>K^i(b_j \mid j\in J)\>\subseteq\> F^i\>.
\]
We conclude that
\[
L.F\>\subseteq\> F^i\>,
\]
so the extension $(L.F|F,v)$ is unramified.

Now we prove the assertion in the general case, where $vL$ is not necessarily a subgroup of $vF$. We construct an
extension $(F_1,v)$ of $(F,v)$ within its absolute ramification field $(F^r,v)$ such that $vF_1=vL+vF$. Take
$(F_1,v)$ to be a maximal extension of $(F,v)$ within $(F^r,v)$ such that $vF_1\subseteq vL+vF$; this exists by
Zorn's Lemma. We have to show that $vF_1=vL+vF$. Suppose otherwise and take an element $\beta\in vL\setminus
vF_1\,$. Let $n$ be the order of $\beta$ over $vF_1\,$; as it must be a divisor of the order of $\beta$ over
$vK$ and $(L,v)$ lies in the absolute ramification field of $(K,v)$, it is not divisible by $\chara Kv$. It
follows that $\beta\in vF_1^r$. Take an element $c\in F_1$ such that $vc=n\beta$.
%
%
Then by Lemma~\ref{l} there is some $b\in (F_1^r)^i=F_1^r=F^r$ such that $b^n=c$ and therefore, $vb=\beta$.
We compute:
\[
n\>=\>(vF_1+\Z\beta:vF_1)\>\leq\>(vF_1(b):vF_1)\>\leq\>[F_1(b):F_1]\>\leq\>n\>,
\]
so equality holds everywhere and we find that $vF_1(b)=vF_1+\Z\beta\subseteq vL+vF$. Since $b\notin F_1\,$, this
contradicts the maximality of $F_1\,$, showing that $vF_1=vL+vF$.

Now we apply what we have shown already to $F_1$ in place of $F$. Since now $vL\subseteq vF_1\,$, we find that
$L.F_1\subseteq F_1^i\subseteq F_1^r=F^r$ and
\[
v(L.F)\>\subseteq\>v(L.F_1)\>\subseteq\>vF_1^i\>=\>vF_1\>=\>vL+vF\>\subseteq\>v(L.F)\>,
\]
whence $v(L.F)=vL+vF$.
\pars
Assume that $vL$ is not a subgroup of $vF$. Then $vF\subsetneq vL+vF=v(L.F)$, so the extension $(L.F|F,v)$ is not
unramified. We have now proved part 1) of Theorem~\ref{GAL}.

\parm
For the proof of part 2) of Theorem~\ref{GAL}, we proceed in a similar way as for part 1), but on a
``lower level''. By hypothesis, $L\subseteq K^i$.
First, we assume that $Lv\subseteq Fv$. We take a set of generators $\{\zeta_j\mid j\in J\}$
of the separable-algebraic field extension $Lv|Kv$. Then we choose monic
polynomials $f_j\in K[X]$ such that the reduction $\ovl{f}_j$ of $f_j$
modulo $v$ is the minimal polynomial of $\zeta_j$ over $Kv$, for each
$j\in J$. Since $\zeta_j$ is a simple root of $\ovl{f}_j$, we can
use Hensel's Lemma to find a root $b_j\in L^h$ whose residue is
$\zeta_j\,$. Since $K^h\subseteq L^h$, we have that $K^h(b_j\mid j\in J)\subseteq L^h$ and
\[
Lv\>\subseteq\>K^h(b_j\mid j\in J)v\>\subseteq \> L^hv\>=\>Lv\>,
\]
showing that equality must hold. We also have that
\[
vL \>\subseteq\> vK^i\>=\>vK\>\subseteq\>vK^h(b_j\mid j\in J)\>\subseteq\>vL^h\>=\>vL\>,
\]
showing again that equality must hold. Thus, $(L^h|K^h(b_j\mid j\in J),v)$ is an immediate extension of henselian
fields inside of the absolute inertia field of $(K,v)$. Hence by Corollary~\ref{cordl} we obtain that
\[
L^h\>=\>K^h(b_j\mid j\in J)\>.
\]

Since $K\subseteq F$, it follows that $K^h\subseteq F^h$. Since
$\zeta_j \in Fv$ and $\zeta_j$ is a simple root of $\ovl{f}_j\,$, it
follows from Hensel's Lemma that $f_j$ has a root in $F^h$ with residue
$\zeta_j\,$; this root must be $b_j\,$. Consequently,
\[
L\>\subseteq\> L^h\>=\>K^h(b_j\mid j\in J)\>\subseteq\> F^h\>.
\]
We conclude that
\[
L.F\>\subseteq\> F^h\>,
\]
which implies that the extension $(L.F|F,v)$ is immediate.

\pars
Next, we prove the assertion in the general case, where $Lv$ is not necessarily a subfield of $Fv$. We construct an
extension $(F_1,v)$ of $(F,v)$ within its absolute inertia field $(F^i,v)$ such that $F_1v=Lv.Fv$. Take
$(F_1,v)$ to be a maximal extension of $(F,v)$ within $(F^i,v)$ such that $F_1v\subseteq Lv.Fv$; this exists by
Zorn's Lemma. We have to show that $F_1v=Lv.Fv$. Suppose otherwise and take an element $\zeta\in Lv\setminus
F_1v$. Since $(L,v)$ lies in the absolute inertia field of $(K,v)$ by hypothesis, $\zeta$ is separable-algebraic
over $Kv$ and hence also over $F_1v$. It follows that $\zeta\in F_1^i v$. Take a monic polynomial $f\in F_1[X]$
whose reduction $fv$ modulo $v$ is the minimal polynomial of $\zeta$ over $F_1v$ and note that $\zeta$ is a simple
root of $fv$. By Hensel's Lemma there is a root $z$ of $f$ in the henselian field $(F_1^i,v)$ such that $zv=
\zeta$. We compute:
\[
\deg f\>=\> \deg fv\>=\> [F_1v(\zeta):F_1v]\>\leq\> [F_1(z)v:F_1v]\>\leq\>[F_1(z):F_1]\>\leq\>\deg f\>,
\]
so equality holds everywhere and we find that $F_1(z)v=F_1v(\zeta)\subseteq Lv.Fv$. Since $z\notin F_1\,$, this
contradicts the maximality of $F_1\,$, showing that $F_1v=Lv.Fv$.

Now we apply what we have shown already to $F_1$ in place of $F$. Since now $Lv\subseteq F_1v\,$, we find that
$L.F_1\subseteq F_1^h\subseteq F_1^i=F^i$ and
\[
(L.F)v\>\subseteq\>(L.F_1)v\>=\>F_1^h v\>=\>F_1v\>=\>Lv.Fv\>\subseteq\>(L.F)v\>,
\]
whence $(L.F)v=F_1v=Lv.Fv$.

Finally, assume that $Lv$ is not a subfield of $Fv$. Then $Fv\subsetneq Lv.Fv=(L.F)v$, so the extension
$(L.F|F,v)$ is not immediate. We have now proved part 2) of Theorem~\ref{GAL}.

\bn
%
%
\section{Proof of Theorem~\ref{MT2}}            \label{sectpfMT2}
By assumption, $\chara Kv=p>0$. We let $L_i$, $L_r$ and $L_s$ be as introduced before Theorem~\ref{brt}. Since
$vL/vL_r$ is a $p$-group and no element of $vL^r/vK$ has order divisible by $p$, we have that $vL_r=(vL)_{p'}$.
Further, $L^i=L.K^i$ is a normal extension of $K^i$ and $L_s^i=L_s.K^i$ is a
Galois extension of $K^i$, with ramification field $L_r^i=L_r.K^i$; thus, $L_s^i|L_r^i$ is a $p$-extension.
\pars

We know that $L_s|L_r$ is a $p$-extension. By Proposition~\ref{p-ext}, this implies that also
$L_s.F|L_r.F$ is a $p$-extension. Since $L|L_s$ is purely inseparable,
it follows that also $L.F|L_s.F$ is purely inseparable. These two
facts imply that $v(L.F)/v(L_r.F)$ is a $p$-group, and that $(L.F)v/(L_r.F)v$ is a
normal extension with its maximal separable subextension being a $p$-extension. Since $v(L_r.F)=(vL)_{p'}+vF$ by
part 1) of Theorem~\ref{GAL}, the former proves part 1) of Theorem~\ref{MT2}.

\pars
Now assume that $(vL)_{p'}=vK$. This implies that $L_r=L_i$ and $L_r.F=L_i.F$. Hence from part 2) of
Theorem~\ref{GAL} it follows that $(L_r.F)v=(L_i.F)v=(Lv)_s.Fv$. Together with the facts about $(L.F)v/(L_r.F)v$
that we showed above, this proves part 2) of Theorem~\ref{MT2}.

\bn
%
%
\section{A closer analysis of the relevant ramification theory}    \label{sectrt}
Throughout this section we will assume that $L|K$ is a (not necessarily finite) Galois extension. Then also
$L.F|F$ is a Galois extension, and we denote by $\res$ the restriction of automorphisms in $\Aut L.F|F$ to $L$.
The following is a consequence of \cite{Neu} (see also \cite{Kubook}).
\begin{proposition}                             \label{res}
In the above situation, we have:
\pars
$\res G^d(L.F|F,v)\>\subseteq G^d(L|K,v)\>,$ \par
$\res G^i(L.F|F,v)\>\subseteq G^i(L|K,v)\>,$ \par
$\res G^r(L.F|F,v)\>\subseteq G^r(L|K,v)\>.$
\end{proposition}

We set $E:=L.F$, let $L_d$, $L_i$ and $L_r$ be as introduced before Theorem~\ref{brt}, and correspondingly
denote by $E_d\,,\,E_i\,,\,E_r$ the decomposition, inertia and
ramification field, respectively, of $(E|F,v)$. As a consequence of Proposition~\ref{res}, we obtain:
\begin{proposition}                             \label{res2}
With the above assumptions and notation, we have that
\[
L_d\subseteq E_d\cap L\>, \qquad L_i\subseteq E_i\cap L\>, \qquad L_r\subseteq E_r\cap L
\]
and
\[
L_d.F\subseteq E_d\>, \qquad L_i.F\subseteq E_i\>, \qquad L_r.F\subseteq E_r\>.
\]
\end{proposition}

We wish to give examples that show that the inclusion may be strict, even if $F|K$ is finite. In fact, this
phenomenon occurs in all instances of elimination of tame or wild ramification.
\begin{example}
{\rm We build on a famous example for an extension with nontrivial defect (see, e.g., \cite{Kudef}). We take
$(K,v)$ to be the perfect hull of the Laurent series field $\F_p((t))$ over the field $\F_p$ with $p$ elements.
We let $\vartheta$ be a root of the Artin-Schreier polynomial $X^p-X-1/t$. As $(K,v)$ is henselian, there is a
unique extension of $v$ to $K(\vartheta)$. Then $(K(\vartheta)|K,v)$ is an immediate Galois extension of degree
$p$, hence has nontrivial defect. The same is true for the extension $(K(\vartheta+a)|K,v)$ where $a$ is a root of
$X^p-X-1$. We set $L=K(\vartheta)$ and $F=K(\vartheta+a)$. We obtain that $L.F=F(a)$. Since $\F_p(a)|\F_p$ is a
separable extension of degree $p$, we see that $L.F=(L.F|F,v)^i$. But as $(K(\vartheta)|K,v)$ has nontrivial
defect, $(K(\vartheta),v)$ does not lie in $K^r$, and consequently, $L_r=K$. With the notation introduced above,
we conclude that $K=L_d=L_i=L_r\subsetneq L\,$, but $F=E_d\subsetneq E_i=E_r=E$ and therefore, $F=L_i.F\subsetneq
E_i$ and $F=L_r.F \subsetneq E_r\,$. }       \diams
\end{example}
\n
This example shows that the $p$-extension mentioned in part 2) of Theorem~\ref{MT2} can be nontrivial even if
$Lv=(Lv)_s=Kv$ and hence $(Lv)_s.Fv=Fv$. In this example, we have in fact eliminated wild ramification, since
$E_r=E$; the wild ramification was turned into a tame unramified extension. It should be noted at this point that
eliminating wild ramification cannot increase tame ramification:
\begin{remark}
If $E_r=E$, then $vE=(vL)_{p'}+vF$. This follows from part 1) of Theorem~\ref{MT2} which states that $vE/
((vL)_{p'}+vF)$ is a $p$-group. But as no element in $vE_r/vF$ has a order divisible by $p$, the group $vE/
((vL)_{p'}+vF)$ must be trivial.
\end{remark}

\pars
The next example is a basic example of the elimination of tame ramification:
\begin{example}
{\rm We take $K=k(t,x)$ and $v$ to be the $t$-adic valuation on $K$. Then $vK=\Z$ and $Kv=k(x)$. We choose an
integer $n>1$ which is not divisible by $\chara k$, and $n$-th roots $t^{1/n}$ and $x^{1/n}$ of $t$ and $x$,
respectively. We assume that $k$ contains a primitive $n$-th root of unity and set $L=K(t^{1/n})$ and
$F=K(t^{1/n}x^{1/n})$, so that $L.F=F(x^{1/n})=(L.F|F,v)^i$. In this
situation, we have that $K=L_d=L_i\subsetneq L_r=L\,$, but $F=E_d\subsetneq E_i=E_r=E$ and therefore,
$F=L_i.F\subsetneq E_i$ and $F\subsetneq L_r.F = E_i\,$. }       \diams
\end{example}

Finally, we give an example where a separable extension of the residue field is eliminated. This corresponds to a
well known procedure using Hensel's Lemma within the henselization of $(F,v)$.
\begin{example}
{\rm We take $(K,v)$ to be as in the previous example, assuming in addition that $\chara Kv=p>0$. We let $a$ be a
root of the Artin-Schreier polynomial $X^p-X-x$, and $b$ a root of $X^p-X-x-t$. We set $L=K(a)$ and $F=K(b)$. We
obtain that $L.F=F(b-a)$. Since $b-a$ is a root of the polynomial $X^p-X-t$ and $vt>0$, $b-a$ lies in the
henselization of $(F,v)$ and it follows that $L.F=E_d\,$. In this situation, we have that $K=L_d\subsetneq L_i=
L_r=L\,$, but $F\subsetneq E_d=E_i=E_r=E$ and therefore, $F\subsetneq L_i.F=E_d=E\,$.}       \diams
\end{example}

\bn
%
%
\section{Examples with rational function fields $F=K(x)$}                           \label{sectex}
\begin{example} \rm
We take a valued field extension $(K(a)|K,v)$ such that $a^n\in K$, the order of $va$ modulo $vK$ is $n$ and $n$ is
not divisible by $\chara Kv$. It follows that $vK(a)=vK+\Z va$ and $K(a)v=Kv$. We set $L:=K(a)$. Further, we
consider the Gau{\ss} valuation $v$ on the rational function field $L(y)$, that is,
\[
v\sum_{i=0}^{k} a_i y^i \>:=\> \min \{va_i\mid 0 \leq i \leq k \}\>.
\]
We choose some $d\in K$ such that $vd>va$ and set $x:=a+dy$, so $K(x)$ is a rational function field contained in
$L(y)$. We consider $K(x)$ equipped with the restriction of the valuation $v$ of $L(y)$.

We wish to prove that $L\subset K(x)^h$. We observe that $x/a$ and $x^n/a^n$ are 1-units and that $x/a$ is a root
of the polynomial
\begin{equation}                             \label{pol}
X^n-\frac{x^n}{a^n}\,\in K(x)[X]
\end{equation}
whose reduction modulo $v$ is $X^n-1$. Since $n$ is not divisible by $\chara Kv$, $1$ is a simple root of this
polynomial and Hensel's Lemma shows that $K(x)^h$ contains a unique root $z$ of (\ref{pol}) with residue $1$.
Consequently, $z=x/a$, whence $a=x/z\in K(x)^h$. This proves that $L\subset K(x)^h$.       \diams
\end{example}

Modifications of this example can be obtained by choosing different extensions of $v$ from $L$ to $L(y)$. For
example, one can define
\begin{equation}                                  \label{extv}
v\sum_{i=0}^{k} a_i y^i \>:=\> \min \{va_i +ivd\mid 0 \leq i \leq k \}\>.
\end{equation}
where again $d\in K$ with $vd>va$. In this case we set $x:=a+y$ and proceed as in the example. Note that in both
constructions, $K(x)v$ is transcendental over $Kv$; in this case the extensions $(K(x)|K,v)$ are called
\bfind{residue transcendental}. In the example, we have that $L(y)v=Lv(yv)=Kv(yv)$ is
transcendental over $Kv$ and since $L(x)|K(x)$ is algebraic, the same must be true for $K(x)v$. In the modified
construction we have that $L(y)v=Lv((y/d)v)=Kv((y/d)v)$.

\parm
A similar example can be produced with a \bfind{value transcendental} extension $(K(x)|K,v)$ where
$vK(x)/vK$ has rational rank $1$. To achieve this, one replaces $vd$ in definition (\ref{extv}) by some value
$\alpha>va$ which is non-torsion over $vK$. A particular case of this is obtained when one takes $v_y$ to be the
$y$-adic valuation on $L(y)$ and then sets the composition $v_y\circ v$ to be the extension of $v$ from $L$ to
$L(y)$.

\parm
In all of the above examples the extension $(K(a)|K,v)$ was such that $vK(a)=vK+\Z va$ and $K(a)v=Kv$. However,
the examples work in exactly the same way when we assume that $a^n\in K$, $va=0$, $[Kv(av):Kv]=n$ and $n$ is
not divisible by $\chara Kv$. It then follows that $vK(a)=vK$ and $K(a)v=Kv(av)$. In this case it is not
tame ramification that is eliminated, but a separable-algebraic extension of the residue field instead.

\bn
%
%
\section{Abhyankar's Lemma using ramification indices}           \label{sectALri}
Theorem~\ref{Abhyankar lemma} is a consequence of the more general version of Abhyankar's Lemma stated in
Lemma 15.102.4 [\href{https://stacks.math.columbia.edu/tag/0EXT}{Tag 0EXT}] of \cite{SP}. Indeed, in the setup
of Lemma 15.102.4 [\href{https://stacks.math.columbia.edu/tag/0EXT}{Tag 0EXT}] and Remark 15.102.1
[\href{https://stacks.math.columbia.edu/tag/0EXT}{Tag 0EXT}], we note that the assumptions that gcd$(e,p) = 1$
and $\kappa_B / \kappa_A$ is separable still hold when the valued field extension $L/K$ is tamely ramified.
Further, $A_1$ is a discrete valuation ring of rank 1 by $(4)$ of Remark 15.102.1
[\href{https://stacks.math.columbia.edu/tag/0EXT}{Tag 0EXT}]. Finally, from Definition 15.109.1
[\href{https://stacks.math.columbia.edu/tag/0ASF}{Tag 0ASF}] and Lemma 15.99.5
[\href{https://stacks.math.columbia.edu/tag/09E7}{Tag 09E7}], it follows that the formally smooth conclusion in
Lemma 15.102.4 [\href{https://stacks.math.columbia.edu/tag/0EXT}{Tag 0EXT}] implies that the extension is
unramified.

\parm
We will now show how Theorem~\ref{GALrr1} can be deduced from Theorem~\ref{GAL}. We will need the following
preparation. If $\Delta$ is a torsion free abelian group and $e>0$ is an integer, then $\frac 1 e \Delta$ will
denote the abelian group consisting of all $\alpha$ in the divisible hull of $\Delta$ such that $e\alpha\in\Delta$.
\begin{lemma}                             \label{ratgr}
Take an integer $e>0$, a torsion free abelian group $\Delta$ of rational rank 1, and a subgroup $\Gamma$ of its
divisible hull such that $\Delta\subseteq\Gamma$ and $(\Gamma:\Delta)=e$. Then $\Gamma=\frac 1 e \Delta$.
\end{lemma}
\begin{proof}
As $\Delta$ of rational rank 1, it can be embedded in $\Q$ by sending any nonzero element in $\Delta$ to $1$, and
the divisible hull of $\Delta$ can be identified with $\Q$. As $(\Gamma:\Delta)=e$, we have that $\Gamma\subseteq
\frac 1 e \Delta$. We wish to show that $(\frac 1 e \Delta:\Delta)=e$, which then yields that $\Gamma=\frac 1 e
\Delta$. It suffices to show that $(\frac 1 e \Delta:\Delta)\leq e$.

Take any $e+1$ many elements $\alpha_1,\ldots,\alpha_{e+1}\in \frac 1 e \Delta$; we have to show that at least
two of them have the same coset modulo $\Delta$. As these elements are rational numbers, we can multiply them by
a common denominator $s$ to obtain integers $s\alpha_1,\ldots,s\alpha_{e+1}\,$. The ideal they generate in $\Z$
is principal, equal to, say, $r\Z$. We know that $(\Z:e\Z)=e$ and hence also $(r\Z:er\Z)=e$. Thus there are
distinct $i,j\in\{1,\ldots,e+1\}$ such that $s\alpha_i-s\alpha_j\in er\Z$. This implies that $\alpha_i-\alpha_j\in
e\frac r s \Z$. Since the elements $s\alpha_1,\ldots,s\alpha_{e+1}$ generate the group $r\Z$, the elements
$\alpha_1,\ldots,\alpha_{e+1}$ generate the group $\frac r s \Z$, which shows that $\frac r s \Z\subseteq
\frac 1 e \Delta$, whence $\alpha_i-\alpha_j\in e\frac r s \Z\subseteq\Delta$. Therefore, $\alpha_i$ and $\alpha_j$
have the same coset modulo $\Delta$.
\end{proof}

As mentioned in the introduction,
the assumption that $(L.K^h|K^h,v)$ is tame yields that $(L.K^h,v)$ lies in the absolute ramification field of
$(K^h,v)$, which is equal to the absolute ramification field of $(K,v)$. Since $vK$ has rational rank 1,
Lemma~\ref{ratgr} shows that the
value group of $(L,v)$ is $\frac{1}{(vL:vK)} vK$, and likewise, the value group of $(F,v)$ is $\frac{1}{(vF:vK)}
vK$. Now we infer from Theorem~\ref{GAL} that
\[
v(L.F) \>=\> \frac{1}{(vL:vK)} vK\>+\> \frac{1}{(vF:vK)}vK\>.
\]
If $\ell$ is the least common multiple of $(vL:vK)$ and $(vF:vK)$, then the right hand side is equal to
$\frac{1}{\ell}vK$. This proves Theorem~\ref{GALrr1}.

\parm
We wish to investigate how far Theorem~\ref{Abhyankar lemma} can be generalized while keeping the use of
ramification indices. We note that if $q$ is a prime and $a,b\in K\ac$ such that $a^q,b^q\in K$, then $va,vb
\in \frac{1}{q}vK$, and that $\frac{1}{q}vK / vK$ is an $\F_q$-vector space.
\begin{lemma}                         \label{2ri}
Take a valued field $(K,v)$ and an extension of $v$ to the algebraic closure $K\ac$ of $K$. Assume that there are
$a,b\in K\ac$ with $va,vb\notin vK$ and a prime $q$ such that $a^q,b^q\in K$ and $va+vK$ and $vb+vK$ are
$\F_q$-linearly independent elements in $\frac{1}{q}vK/vK$.
Then we have that $(vK(a):vK)=q=(vK(b):K)$ and that
\begin{equation}                         \label{vKab}
(vK(a,b):vK(a))\>=\>q\>=\>(vK(a,b):K(b))\>.
\end{equation}
\end{lemma}
\begin{proof}
We compute:
\[
(vK(a):vK)\>\leq\>[K(a):K]\>\leq\>q\>=\>(vK+\Z va:vK)\>\leq\>(vK(a):vK)\>.
\]
Thus, equality holds everywhere, showing that $(vK(a):vK)=q$. In a similar way, one shows that $(vK(b):vK)=q$.
Further, the equality $(vK+\Z va:vK)=(vK(a):vK)$ shows that $vK(a)=vK+\Z va$. Similarly, it is shown that
$vK(b)=vK+\Z vb$. Obviously, $va,vb\in vK(a,b)$. However, since $va+vK$ and $vb+vK$ are $\F_q$-linearly
independent elements in $\frac{1}{q}vK/vK$, we have that $va\notin vK+\Z vb= vK(b)$ and $vb\notin vK+\Z va =
vK(a)$. As $q$ is a prime, we conclude that $(vK(a,b):vK(b))\>\geq\>q$
and $(vK(a,b):vK(a))\>\geq\>q$, and with similar inequalities as above, one proves that (\ref{vKab}) holds.
\end{proof}
\sn
This lemma shows that Theorem~\ref{Abhyankar lemma} will fail as soon as there exist a prime $q$ different from
the residue characteristic and two values $\alpha, \beta\in vK$ such that both are not divisible by $q$ in $vK$
and $\alpha/q+vK$ and $\beta/q+vK$ are $\F_q$-linearly
independent elements in $\frac{1}{q}vK/vK$. Then one can pick $a,b\in
K\ac$ such that $a^q,b^q\in K$ with $va^q=\alpha$ and $vb^q=\beta$. It follows that $a,b\notin K$, so these
elements satisfy the assumptions of Lemma~\ref{2ri}.

Quick examples for the above situation are valued
fields $(K,v)$ for which $vK$ is isomorphic to $\Z^n$ with $n>1$, endowed with any ordering. These include all
generalized discretely valued fields with $n>1$.

\bn
\bn
\bn


\begin{thebibliography}{99}
\bibitem{Abh} {Abhyankar, A.: {\it Ramification theoretic methods in algebraic geometry}, Princeton Univ.\ Press,
1959}

\bibitem{An} {André, Y.: {\it Le lemme d'Abhyankar perfectoide}, Math.\ Inst.\ Hautes \'Etudes Sci.\ {\bf 127}
(2018), 1--70}

\bibitem{CH} {Chabert, J.-L.\ -- Halberstadt, E.: {\it On Abhyankar's lemma about ramification indices},
arXiv:1805.08869v1}

\bibitem{En} {Endler, O.: {\it Valuation theory}, Berlin (1972)}

\bibitem{EP} {Engler, A.J.\ -- Prestel, A.$\,$: {\it Valued fields},
Springer Monographs in Mathematics. Springer-Verlag, Berlin, 2005}

\bibitem{Kurf} {Kuhlmann, F.-V., {\it Value groups, residue fields and bad places of rational function fields},
Trans.\ Amer.\ Math.\ Soc.\ {\bf 356} (2004), 4559-4600.}

\bibitem{Kudef}{Kuhlmann F.-V.: {\it Defect}, in: Commutative Algebra - Noetherian and non-Noetherian perspectives,
Fontana, M., Kabbaj, S.-E., Olberding, B., Swanson, I. (Eds.), Springer 2011}

\bibitem{Kubook} {Kuhlmann, F.-V.: Book in preparation. Preliminary
versions of several chapters available at:
{\tt http://math.usask.ca/$\,\tilde{ }\,$fvk/Fvkbook.htm}}

\bibitem{N} {Narkiewicz, W.: {\it Elementary and analytic theory of algebraic numbers}, Springer 2004}

\bibitem{Neu} {Neukirch, J.: Algebraische Zahlentheorie, Springer Berlin (1990)}

\bibitem{P1} {Ponomarëv, K.~N.: {\it Solvable elimination of ramification in extensions of discretely valued
fields}, Algebra i Logika {\bf 37} (1998), 63--87, 123; translation in Algebra and Logic {\bf 37} (1998), 35--47}

\bibitem{P2} {Ponomarëv, K.~N.: {\it Some generalizations of Abhyankar's lemma}, Algebra and model theory,
(Erlagol, 1999), 119--129, 165, Novosibirsk State Tech.\ Univ., Novosibirsk, 1999}

\bibitem{SP} {Stacks Project, http://stacks.math.columbia.edu, 2018}

\bibitem{ZS2} {Zariski, O.\ -- Samuel, P.: {\it Commutative
Algebra}, Vol.\ II, New York--Heidel\-berg--Berlin (1960)}
\end{thebibliography}
\end{document}